\documentclass[12pt]{amsart}

\usepackage{amscd}
\usepackage{amsmath}
\usepackage{amssymb}
\usepackage{amsthm}
\usepackage{bigdelim}
\usepackage{color}
\usepackage{enumerate}
\usepackage{mathrsfs}
\usepackage{multirow}
\usepackage[all]{xy}
\newtheorem{thm}{Theorem}[section]
\newtheorem{cor}[thm]{Corollary}
\newtheorem{prop}[thm]{Proposition}
\newtheorem{conj}[thm]{Conjecture}

\newtheorem{lem}[thm]{Lemma}
\theoremstyle{definition}
\newtheorem{defn}[thm]{Definition}

\newtheorem{eg}[thm]{Example}
\newtheorem{ques}[thm]{Question}

\theoremstyle{remark}
\newtheorem{rem}[thm]{Remark}

\newtheorem*{ack}{Acknowledgements}

\newcommand{\qtq}[1]{\quad\mbox{#1}\quad}
\newcommand{\tsum}[0]{\textstyle{\sum}}
\newcommand{\spec}{\mathrm{Spec}\, }

\title[Subadditivity of Kodaira dimension does not hold in char $p>0$]{Subadditivity of Kodaira dimension does not hold in positive characteristic}
\author{Paolo Cascini}
\address{Department of Mathematics, Imperial College London, 180 Queen's Gate,
London SW7 2AZ, UK}
\email{p.cascini@imperial.ac.uk}

\author{Sho Ejiri}
\address{Department of Mathematics, Graduate School of Science, Osaka University,
Toyonaka, Osaka 560-0043, Japan}
\email{shoejiri.math@gmail.com, s-ejiri@cr.math.sci.osaka-u.ac.jp}

\author{J\'{a}nos Koll\'{a}r}
\address{Department of Mathematics, Fine Hall, Washington Road,
Princeton University,  Princeton, NJ 08544-1000  USA}
\email{kollar@math.princeton.edu}

\author[Lei Zhang]{Lei Zhang}
\address{School of Mathematical Sciences,
	University of Science and Technology of China, Hefei, 230026, P. R. of China}
\email{zhlei18@ustc.edu.cn}

\begin{document}
\maketitle
\begin{abstract}
Over any algebraically closed field of positive characteristic, we construct  examples of fibrations violating subadditivity of Kodaira dimension.
\end{abstract}

\section{Introduction}
Throughout this note, a \emph{fibration} is a projective morphism $f\colon X \rightarrow Y$ between two normal varieties such that $f_*\mathcal{O}_X \cong \mathcal{O}_Y$.
For fibrations between varieties defined over the field $\mathbb{C}$ of complex numbers, Iitaka~\cite{Ii72} proposed the following conjecture:
\begin{conj}[Iitaka' conjecture~\cite{Ii72}] \label{conj:iitaka}
Let $f\colon X\rightarrow Y$ be a fibration between  smooth projective varieties over $\mathbb{C}$, with $\dim X = n$ and $\dim Y = m$.
Let $X_{\bar{\eta}}$ denote the geometric generic fibre of $f$.
Then the following inequality holds:
 \begin{equation}\tag{$C_{n,m}$}
 \kappa(X)\geq \kappa(Y) + \kappa(X_{\bar{\eta}}).
 \end{equation}
\end{conj}

Considering the recent developments of birational geometry over a field of positive characteristic, it is natural to ask if Conjecture~\ref{conj:iitaka} can be generalised to this setting. It is important to notice though that, given a fibration $f\colon X\to Y$,  even assuming that both $X$ and $Y$ are smooth, the geometric generic fibre $X_{\bar{\eta}}$ is possibly singular and not even reduced if $f$ is inseparable. Thus Iitaka' conjecture can be posed in two different versions (cf. \cite[Conjecture 1.2 and 1.4]{Zh19a}):

\begin{ques} \label{iitaka-char-p}
Let $f\colon X\rightarrow Y$ be a fibration between normal projective varieties over an algebraically closed field $k$ of positive characteristic,
with generic fibre $X_{\eta}$ and geometric generic fibre $X_{\bar\eta}$. Do the following inequalities hold?
\begin{enumerate}
\item (Strong form)  $\kappa(X)\geq \kappa(Y) + \kappa(X_{\eta})$.
\item (Weak form)  $\kappa(X)\geq \kappa(Y) + \kappa(X_{\bar \eta})$.
\end{enumerate}
\end{ques}

Since in positive characteristic resolution of singularities has been proved only in dimension not greater than three \cite{CP08,CP09}, we  use a notion of Kodaira dimension that is independent of the existence of   resolutions, see Definition~\ref{kappa.defn}. 

Our main result is a series of counterexamples to the Strong form of 
Question~\ref{iitaka-char-p}.

\subsection{Previous work on Iitaka' conjecture}
According to Mori~\cite[\S 6]{Mo87},
before the work of  Ueno~\cite{Ue77} and Viehweg~\cite{Vie77} appeared,
Conjecture ~\ref{conj:iitaka} was widely believed to be false,
probably because of the existence of non-algebraic counterexamples (cf. \cite[Remark~4]{NU73}). 

Over the last few decades, many important cases of the conjecture were solved \cite{Vie77,Vie82,Vie83,Ka81,Ka82,Ka85,Ko87,Fuj03,Bir09,CH11,CP17,Cao18, HPS}, and, in particular, it is known to hold true if (1) $\dim Y=1$ or $2$ \cite{Ka82, Cao18}; (2) $X_{\bar\eta}$ has a good minimal model \cite{Ka85}; (3) $X_{\bar\eta}$ is of general type \cite{Ko87}; (4) $Y$ is of maximal Albanese dimension \cite{CH11,CP17,HPS}.


The positive characteristic version  has been answered affirmatively for fibrations with smooth geometric generic fibres with relative dimension one or with $\dim X=3$ and $p>5$ \cite{CZ15, EZ16}.
One of the advantages of assuming that $X_{\bar{\eta}}$ is smooth lies in the fact that,
under some additional conditions, $f_*\omega_{X/Y}^N$ is weakly positive \cite{Pa14, Ej17}.
Question~\ref{iitaka-char-p} has also been studied without assuming that $X_{\bar\eta}$ is smooth.
Patakfalvi proved $C_{n,m}$ when $Y$ is of general type and $X_{\bar\eta}$ has non-nilpotent Hasse--Witt matrix (\cite{Pa18}).
Assuming that $K_X$ is nef and some other additional conditions, the fourth author \cite{Zh19a} proved that $F_Y^{e*}(f_*\omega_{X/Y}^N\otimes \omega_Y^{N-1})$ contains a non-zero weakly positive subsheaf. This plays a similar role as $f_*\omega_{X/Y}^N$ does when studying Iitaka' conjecture. The fourth author studied fibrations with singular geometric fibres and  proved subadditivity of Kodaira dimension for three-folds over an algebraically closed field of characteristic $p>5$, assuming that  $Y$ is of maximal Albanese dimension, which implies the abundance conjecture for minimal three-folds $X$ with $q(X) >0$ \cite{Zh17, Zh19b}.
The arguments in \cite{Zh17} heavily rely on results of the minimal model program \cite{HX15, Bir13, BW17}, among which, an important ingredient is the fact that the relative minimal model over an abelian variety is actually minimal, as a consequence of the cone theorem \cite[Theorem 1.1]{BW17}.

\begin{defn}\label{kappa.defn}
 Let $X$ be a normal projective variety over a field $K$ which is not necessarily algebraically closed and let $D$ be an integral divisor on $X$. The \emph{$D$-dimension} $\kappa(X,D)$ is defined as
\[\kappa(X,D) =\left\{
\begin{array}{llr}
-\infty  &\text{if }  |mD| = \emptyset \text{ for all $m\in \mathbb Z_{>0}$,}\\
\max \{\dim_k \Phi_{|mD|}(X) \mid m \in \mathbb{Z}_{>0} \} \quad&\text{otherwise.}
\end{array}\right.
\]
Let $K_X$  denote the Weil divisor corresponding to the dualizing sheaf $\omega_X$, namely, $\mathcal{O}_X(K_X) \cong \omega_X$. If $\sigma\colon Z \to X$ is a birational morphism of normal projective  varieties, then
$\sigma_*|mK_Z| \subset |mK_X|$, hence
$\kappa(Z, K_Z) \leq \kappa(X, K_X)$.
We can thus define the \emph{Kodaira dimension} of an irreducible proper $K$-scheme $X$ to be
$$\kappa(X):= \min \{\kappa(Z, K_Z)\mid Z~\mathrm{is~a~projective~normal~birational~model~of}~
\operatorname{red}(X)\}.$$
It is easy to see that $\kappa(X)$ is a birational invariant and coincides with the classical definition of Kodaira dimension if $X$ admits a smooth birational model.

For a fibration $f\colon X\rightarrow Y$ between normal projective varieties, the generic fibre $X_{\eta}$ is normal, hence $\kappa(X_{\eta})$ is always  well defined and equals $\kappa(X_{\bar{\eta}})$ if $X_{\eta}$ is smooth over $k(\eta)$.
\end{defn}


%
\medskip

\subsection{Description of the paper}
 We present two constructions of counterexamples to the strong form of Question~\ref{iitaka-char-p}. The original counterexamples are due to the authors PC, SE and LZ;  these are  global in nature. We describe them in Section~\ref{sec:example}. 
Analysing the generic fibre of the global examples  led JK to more local ones.
 We describe them in Section~\ref{sec.3}. 
We decided to write these up together since, although ultimately they are about the same objects, the two approaches  explore very different aspects of it.

The first one stems from the fibre products of a Raynaud's surface $f\colon S \to C$ over a Tango curve $C$, which satisfies that $f_*\omega_{S/C}^l$ are anti-ample for all $l\geq 1$. We show that, for sufficiently large $m$, the fibre product $X^{(m)}$ has negative Kodaira dimension by verifying that  $f^{(m)}_*\omega_{X^{(m)}}^l \cong (\bigotimes^mf_*\omega_{S/C}^l) \otimes \omega_C^l $ are anti-ample, where $f^{(m)}\colon X^{(m)}:= S\times_C S \times_C \cdots \times_C S \to C$ is the induced morphism. Then for a suitable choice of $m_0$, the fibrations  $X^{(m_0)} \to X^{(l)}$ for $l <m$ provide the desired examples.

The second construction is obtained by considering the product $G^n$ of a regular, but not smooth, curve $G$ defined over the function field $K=k(t)$, with arithmetic genus $p_a(G) \geq 1$. We show that for some sufficiently large integer $n$, the product $G^n$ is birational to a Fano variety. We may do a base change and obtain a fibration $f\colon S \to C$ from a surface $S$ to a curve $C$ of genus $g(C)\ge 2$ and such that $\kappa(S) \geq 0$ and the generic fibre $F$ is isomorphic to $G\otimes_KK(C)$. If, as above $X^{(m)}$ denotes the $(m+1)$-dimensional fibre product of $S$ over $C$ then, for a suitable value of $m_0$ when the generic fibre of $X^{(m_0)} \to C$ is birational to a Fano variety, we have that $X^{(m_0)}$ has negative Kodaira dimension, so that we obtain a desired example similarly as above.

We will see that the first class of examples are special cases of the second class. We present both the two constructions because the computations are different. Indeed, the first construction requires that $C$ is a Tango-Raynaud curve while the computation is easy and direct, and the second construction is more general while the computation is subtle but reveals more information and yields examples of lower dimension. More precisely, for every characteristic $p\geq 3$, we get counterexamples of dimension $\leq 2p-1$, and for $p=2$ we get examples of dimension four.

In both sets of examples the base spaces of the  fibrations are uniruled varieties and the  geometric generic fibres are singular. 
It is still reasonable to expect that subadditivity of Kodaira dimension holds for fibrations in characteristic $p>0$ if, either we assume that the base is not uniruled (cf. \cite{Ej19}), or we replace $\kappa(X_{\eta})$ by the Kodaira dimension $\kappa(Z)$ of a smooth birational model $Z$ of $X_{\bar\eta}$ (cf. \cite{CZ15}). Note that $\kappa(X_{\eta})\ge\kappa(Z)$ (cf. \cite{Ta18}).

\smallskip

At the end of this note (see Section \ref{sec:non-alg-closed-field}), we describe an example due to  Tanaka \cite{Ta16} to show that a logarithmic version of subadditivity of Kodaira dimension does not hold true for some surfaces over non-algebraically closed fields.



\smallskip

\begin{ack}
The authors are grateful to H. Tanaka for helpful comments and pointing out about the log version of Question~\ref{iitaka-char-p} (Example~\ref{eg:tanaka}). We would also like to thank F. Bernasconi, O. Fujino, T.~Murayama, Zs. Patakfalvi, R.~van~Dobben~de~Bruyn, J. Waldron, and C. Xu   for many useful comments and corrections.

The first author was supported by EPSRC.
The second author was supported by JSPS KAKENHI Grant Number 18J00171.
The third author was supported by the NSF under grant number
DMS-1901855.
The fourth author was supported by a NSFC grant (No. 11771260).
\end{ack}

\section{The first construction}\label{sec:example}
 We shall present the first construction in this section stemming from Raynaud's surfaces. For the convenience of the reader, we include all the details.

\subsection{Tango--Raynaud curves} \label{section:TR-curve}
We begin by recalling the notion of a Tango--Raynaud curve, introduced in \cite{Mu13}.
A smooth projective curve $C$ is a \textit{Tango--Raynaud curve} if there exists a line bundle $L$ on $C$ such that
$L^p \cong \omega_C$
and the map
$$
H^1\left(C,L^{-1} \right)
\to
H^1\left(C, L^{-p}\right)
$$
induced by the Frobenius morphism is \textit{not} injective.
As explained in \cite[\S 1]{Mu13},  \cite[Lemma~12]{Ta72} implies that $C$ is a Tango--Raynaud curve if and only if
there exists a rational function $f \in K(C)$ such that
$df \ne 0$ and each coefficient in the divisor $(df)$ is divisible by $p$.
\begin{eg}[\textup{\cite[Example~1.3]{Mu13}}] \label{eg:TR-curve}
Fix $e>0$. Let $C \subset \mathbb P^2$ be the plane curve defined by
$$
Y^{pe} -YX^{pe-1} = Z^{pe-1}X,
$$
where $[X,Y,Z]$ are homogeneous coordinate of $\mathbb P^2$.
Let $U = C \cap \{X \neq 0\} \subset \mathrm{Spec}~ k[y,z]$ with $y = \frac{Y}{X}, z = \frac{Z}{X}$. Then $U$ is defined by $y^{pe} - y = z^{pe-1}$ and $C = U \cup \{\infty :=[0,0,1]\}$.
In particular, we have
$$
-dy = -z^{pe-2} dz
$$
on $U$.
Thus, $\omega_C$ is generated by $dz$ on $U$.
Since the degree of $C$ is $pe$, we have $\deg \omega_C = pe(pe-3)$,
and $(dz) = pe(pe-3) \cdot (\infty)$. It follows that $C$ is a Tango--Raynaud curve.
\end{eg}
\subsection{Raynaud and Mukai's construction of algebraic fibre spaces}
\label{section:RM-surface}

We only use a special case of Mukai's construction \cite{Mu13}. Fix a Tango--Raynaud curve $C$ as in Example \ref{eg:TR-curve} and  let $D = e(pe-3) \cdot (\infty)$ so that $K_C=(dz)= pD$. Assume $e$ is prime to $p$.
We can regard $dz$ as an element of $H^0(C, \mathcal{B}^1(-D))$ where, if $F_C$ denotes the Frobenius morphism on $C$, then $\mathcal{B}^1 := d{F_C}_*\mathcal{O}_C$. Thus,  $dz$ induces a non-zero element
$\xi \in H^1(C, \mathcal O_C(-D))$ lying in the kernel of the map
$$
H^1(C, \mathcal O_C(-D)) \to H^1(C, \mathcal O_C(-pD))
$$
induced by $F_C$.

Thus, $\xi$ corresponds to an extension
\begin{align} \label{extension:1} 
0 \to \mathcal O_C(-D)
\to \mathcal E
\xrightarrow{\alpha} \mathcal O_C
\to 0.
\end{align}
We define $P = \mathbb P(\mathcal E)$ ($:=\mathrm{Proj}_{\mathcal O_C}\bigoplus_{i=0}^{+\infty}S^i\mathcal E$)
and let $g\colon P \to C$ denote the natural projection.
Let $T$ be a divisor corresponding to $\mathcal O_P(1)$ and let $F \subset P$ be the section of $g$ corresponding to $\alpha$.
Then $T|_F \sim 0$.
Since $T-F$ is relatively trivial over $C$,
applying $g_*$ to the exact sequence
\begin{align}
0 \to \mathcal O_P(T -F) \to \mathcal O_P(T) \to \mathcal O_F \to 0,
\end{align}
we obtain (\ref{extension:1}), which implies that $T-F \sim -g^*D$. Thus,
\begin{align} \label{F and T}
F \sim T +g^*D~\text{and}~F^2=\deg D.
\end{align}
If we set $P_1:= \mathbb P(F_C^*\mathcal E)$, then we have the following commutative diagram:
\begin{align}
\xymatrix{ P^1 \ar[dr]^{F_P} \ar[d]_{F_{P/C}} &  \\
P_{1} \ar[r]_{W} \ar[d]_{g_{1}}& P \ar[d]^{g} \\
C^1 \ar[r]_{F_C} & C
}
\end{align}
where $F_{P/C}$ denotes the relative Frobenius morphism.
By the choice of $\xi$, we see that the pull-back
\begin{align} \label{extension:2} 
0 \to \mathcal O_C(-pD)
\to F_C^*\mathcal E
\to \mathcal O_C
\to 0.
\end{align}
of (\ref{extension:1}) by $F_C$ splits.
The splitting map $F_C^* \mathcal E \twoheadrightarrow \mathcal O_C(-pD)$
induces a section $G' \subset P_{1}$ of
$g_{1}\colon P_{1} \to C^1$.
We then have
$$
G' \sim W^* F -p g_{1}^*D.
$$
Set $G:= {F_{P/C}}^*G'$.
Then we get
\begin{align} \label{G and T}
G \sim pF -p g^* D \sim pT.
\end{align}

We may choose $e$ so that there exists  a positive integer $l$ which divides both  $e$ and $p+1$. Let $e':=\frac e l$, $r:=\frac{p+1}l$ and  $D' := \frac{D}{l} = e'(pe-3)\cdot (\infty)$.
Then
$$
G + F \sim (p+1)T + g^*D = l\left(rT + g^*D'\right ) = lM,
$$
where $M:= rT + g^*D'$. The equivalence yields a cyclic $l$-cover $\pi\colon S \to P$ branched over $G+F$ such that $S$ is smooth.
We use the following notation:
$$
\xymatrix{
S \ar@/_10pt/[rr]_{f} \ar[r]^{\pi} & P \ar[r]^{g} & C.
}
$$
By an easy calculation we have
\begin{equation}\label{canonical bundle formula}
\begin{split}
K_{S/C} &\sim \pi^* (K_{P/C} + (l-1)M)\\
& \sim \pi^* (-2T -g^*D + (l-1)M )\\
& \sim \pi^*\left ( \left (p-1-r\right )T\right) - f^*D'.
\end{split}
\end{equation}
Let $q:=p-1-r$. Then, for all positive integer $n$ such that $nq\ge r(l-1)$, we have
\begin{align*}
f_*\omega_{S/C}^n &\cong g_*\left (\mathcal O_P\left(n\left(p-1-r\right)T -ng^*D'\right)\otimes \bigoplus_{i=0}^{l-1} \mathcal O_S(-iM) \right)  \\
& \cong  \bigoplus_{i=0}^{l-1}g_*\mathcal{O}_P\left(\left(nq - ri\right)T -(n + i)g^*D'\right) \\
& \cong \bigoplus_{i=0}^{l-1}S^{nq-ri}\mathcal{E}(-(n+i)D').
\end{align*}

\begin{lem} \label{lem:K_S ample}
If $\{l,p\} \ne \{2,3\}$, then $K_S$ is ample, otherwise  $\kappa(S) = 1$.
\end{lem}
\begin{proof}
By (\ref{canonical bundle formula}), we have that
\begin{align*}
K_S &\sim \pi^*(qT) - f^*D' + f^*K_C \\
& \sim  \pi^*(qT) - f^*D' + f^*(pD) \\
&= \pi^*(q(T+g^*D) + g^*((p-q)D - D')) \\
& \sim  \pi^*(qF + g^*((p+l)D')).
\end{align*}
Note that $F$ is nef and big because $F^2 >0$. Therefore, if $\{l,p\} \ne \{2,3\}$ then $q>0$ and $K_S$ is ample; otherwise $q = 0$ and then $\kappa(S) =1$.
\end{proof}

\subsection{Fibre products of Tango--Raynaud--Mukai surface} \label{section:fibre_product}
We use the same notation as in Section~\ref{section:RM-surface}.
Pick a positive integer $m$.
Let $X^{(m)}:=S \times_C \cdots \times_C S$ be the $m$-th fibre product of $S$ over $C$.
Then $X^{(m)}$ is a Gorenstein integral scheme, as each $X^{(i)}\to X^{(i-1)}$ is a flat morphism whose every fibre is a Gorenstein (integral) curve.
Denote by $f^{(m)}\colon  X^{(m)} \to C$ the natural fibration and by $p_i: X^{(m)} \to S$ the projection to the $i$-th factor.
We then have
$$
\omega_{X^{(m)}/C} \cong \bigotimes_{i=1}^{m} p_i^*\omega_{S/C}.
$$
Using the projection formula, for any positive integer $n$ such that $nq\ge r(l-1)$, we get
\begin{equation}\label{eq:push-pluri-canonical}
\begin{split}
f^{(m)}_*\omega_{X^{(m)}}^n
&\cong f^{(m)}_*\omega_{X^{(m)}/C}^n\otimes\omega_C^n\\
&\cong \left(\bigotimes^m f_*\omega_{S/C}^n \right)\otimes \omega_C^n\\
&\cong \left( \bigoplus_{0\leq i_1, \cdots, i_m \leq l-1}
\left( \bigotimes_{1\le j\le m} S^{nq-ri_j}\mathcal{E} \right) \right)
\otimes \mathcal{O}_C\left(\left(npl-mn-\sum_{j=1}^m{i_j}\right)D'\right).
\end{split}
\end{equation}

It follows that
\begin{lem}\label{kod-dim-Xm}
For $m >pl$, the dual $(f^{(m)}_*\omega_{X^{(m)}}^n)^{\vee}$ is ample. In particular, $$\kappa(X^{(m)},K_{X^{(m)}})= -\infty.$$
\end{lem}

Note that for $m \geq 2$, the variety $X^{(m)}$ is singular and,
if $Z\subset S$ is the non-smooth locus of $f$ (i.e., $Z=\mathrm{Supp}(G)$),
then $(X^{(m)})_{\mathrm{sing}}=\bigcup_{1\le i<j\le m}p_i^{-1}(Z)\cap p_j^{-1}(Z)$, which is of codimension 2. Moreover the Serre condition $S_2$ is satisfied, hence $X^{(m)}$ is normal. As showed in the Introduction, we have
\begin{thm}\label{kod-dim-Ym}
If $m > pl$  then
$\kappa(X^{(m)}) = -\infty.$
\end{thm}

\subsection{Counterexamples to subadditivity of Kodaira dimension}\label{sec:construction1}
We use the same notation as in the previous sections.
Let $m_0(p)$ be the maximal integer such that $\kappa(X^{(m_0(p))}) \geq 0$.
Lemma \ref{lem:K_S ample} and Lemma \ref{kod-dim-Xm} imply that $1\le m_0(p)\le pl$.

For any $1\le m\le m_0(p)$, we have the following commutative diagram
$$
\xymatrix{&X^{(m_0(p)+1)}\ar[r]\ar[d] & X^{(m_0(p)+1-m)}\ar[d]\\
&X^{(m)}\ar[r]            &C
}
$$
Notice that since $X^{(m)} \to C$ is separable, the generic fibre of $X^{(m_0(p)+1)} \to X^{(m)}$ is isomorphic to  $X^{(m_0(p)+1-m)}_{\eta}\otimes_{K(C)}K(X^{(m)})$ and hence has nonnegative Kodaira dimension.
\medskip

In conclusion,
\begin{itemize}
\item  For $p\geq 3$, we take $l=2$ and $m_0(p) \leq 2p$. The fibration $h\colon Y^{(m_0(p)+1)} \to Y^{(m)}$ with $\dim Y^{(m_0(p)+1)} \leq 2p+2$ violates subadditivity of Kodaira dimension $C_{m_0(p)+2,m+1}$.
\item  For $p=2$, we  take $l=3$ and $m_0(p) \leq 6$. The fibration $h\colon Y^{(m_0(p)+1)} \to Y^{(m)}$ with $\dim Y^{(m_0(p)+1)} \leq 8$ violates subadditivity of Kodaira dimension $C_{m_0(p)+2,m+1}$.
\end{itemize}

\section{The second construction}\label{sec.3}
In this section, we introduce the second construction by considering the product of a regular but not smooth curve over a non-algebraically closed field.
\subsection{Product of an algebraic curve}
Let $k$ be an algebraically closed field of characteristic $p$ and let $K=k(t)$. Fix positive integers $q = p^r$ and $m\ge 2$ such that $p\nmid m$. Denote by $\bar{K}$ the algebraic closure of $K$.

Let $G$ be the regular, projective model of
$$G_x:=\{u^q=x^m+t\} \subset \mathbb{A}_{x,u}^2=\spec K[x,u].$$ Thus, $G$ admits a purely inseparable cover  $\pi \colon G\to \mathbb{P}_K^1 =
\mathbb{A}_x^1 \cup \mathbb{A}_y^1$ where $\mathbb{A}_x^1=\spec K[x]$ and $\mathbb{A}_y^1=\spec K[y]$ with $x = \frac{1}{y}$. In particular, $G$ is obtained by glueing $G_x$
with $G_y$, which is the normalisation of $$\{(1+ty^m)v^q=y^m\}\subset \mathbb A^2_{y,v}=\spec K[y,v]$$ where $v = \frac{1}{u}$. Note that $G_y$ is smooth and $G_x$ is smooth outside $\{x=0\}$.

The curve $G_{\bar{K}}$ is normal outside $\{x =0\}$ and the normalisation of $(G_x)_{\bar{K}}$ is given by
{\small $$(G_x)^\nu_{\bar{K}}:= \spec \bar{K}[z] \to (G_x)_{\bar{K}}= \{u^q=x^m+t\} \subset \spec {\bar{K}}[x,u] \qquad z \mapsto (x=z^q, u=z^m+ t^{\frac{1}{q}}).$$}
It follows that if $a,c$ are positive integers such that $am-cq=1$, then
$$z = x^{-c}(u-t^{\frac{1}{q}})^a.$$

Let $G^n$ be the $n$-th fibre product of $G$ over $\spec K$, and let  $\pi^n\colon G^n \to (\mathbb{P}_K^1)^n$ be the corresponding product map of $\pi \colon G\to \mathbb{P}_K^1$. Note that  $\pi^n$ is purely inseparable.
Let $\mathbf y = (y_1, \cdots, y_n)$ and let $\mathbb A^n_{\mathbf y}=\spec K[y_1,\dots,y_n]\subset \mathbb P_K^n$ where the immersion is given by
$$(y_1, \cdots, y_n) \mapsto [1, y_1, \cdots, y_n].$$
Let $(G_y)^n \subset G^n$ be the pre-image of $\mathbb A^n_{\mathbf y}$ and
let $Y(n,m,q)$ be the normalisation of $\mathbb{P}_K^n$ in $K(G^n)$.
Then $Y(n,m,q)$ admits a finite and purely inseparable cover over $\mathbb{P}_K^n$. We often denote $Y(n,m,q)$ by $Y$ for simplicity. We have  the following commutative diagram
$$\xymatrix{&G^n \ar[d]^{\pi^n} &(G_y)^n\ar@{^{(}->}[r]\ar@{_{(}->}[l]\ar[d] &Y \ar[d]^{\phi} &Y_{\bar{K}}\ar[d]^{\bar{\phi}}\ar[l]^{\mu}  &Z_{\bar{K}} \ar[l]^{\nu}\ar[ld]^{\psi} \\
&(\mathbb{P}_K^1)^n  &\mathbb A^n_{\mathbf y}\ar@{^{(}->}[r]\ar@{_{(}->}[l]   &\mathbb{P}_K^n &\mathbb{P}_{\bar{K}}^n\ar[l] &
}$$
where $\nu\colon Z_{\bar{K}} \to Y_{\bar{K}}$ denotes the normalisation map. Note that the variety $Z_{\bar{K}}$ can also be obtained by normalising $\mathbb{P}_{\bar{K}}^n$ in $K(G_{\bar{K}}^n)$. With the above notation, we have that  $K(G_{\bar{K}}^n) = K(Z_{\bar{K}}) =\bar{K}(z_1, \cdots, z_n)$ and  the morphism $\psi\colon Z_{\bar{K}} \to \mathbb{P}_{\bar{K}}^n$ is determined by the homomorphism between their function fields
$$\psi^*\colon  \bar{K}(y_1, \cdots, y_n) \to \bar{K}(z_1, \cdots, z_n)\qquad y_i \mapsto z_i^{-q}.$$
Therefore, $Z_{\bar{K}} \cong \mathbb{P}_{\bar{K}}^n$.
\smallskip

We now want  to prove that if $n \geq m(q-1)$ then  $Y$ is a Fano variety with canonical singularities. The strategy is to consider the pull-back $\nu^*\mu^*K_Y=\nu^*K_{Y_{\bar{K}}}$ on $Z_{\bar{K}}$, which is equal to $K_{Z_{\bar{K}}}$ plus the conductor. The key point is to compute the conductor. Note that  $Y$ is smooth over the pre-image of $\mathbb A^n_{\mathbf y}$.
\smallskip

Let $\mathcal{Y}$ be the normalisation of $\mathbb{P}_K^n \times \mathbb{A}^1_t$ in $K(G^n)$. Then we get a fibration $\mathcal{Y} \to \mathbb{A}^1_t$ with the generic fibre $Y$. Our goal is to study the singularities of $\mathcal Y$. To this end, given non-zero integers $m_1,\dots,m_{m-1}$ which are relatively prime to $p$, we denote by
$$
\mathbb A^n\bigl\slash \tfrac1{q}(1, m_1,\dots, m_{n-1}) $$
the quotient toric singularity defined by the cone  $\Sigma= \sum_{i=1}^{n} \mathbb{R}_{\geq 0}\cdot e_i\subset N\otimes \mathbb R$, where
$e_1,\dots,e_n$ denote the standard basis in $\mathbb R^n$ and $N$ is  the lattice
$$N:=\tsum_{i=1}^{n}\mathbb{Z}\cdot e_i + \mathbb{Z}\cdot\tfrac1{q}(e_1+ m_1e_2 +\dots, m_{n-1}e_n).$$

\begin{prop}\label{p_only.quotient.say}
\'Etale locally, every singularity of $\mathcal{Y}$
is of the form
\begin{equation}\label{e_quotient}
\mathbb A^n\bigl\slash \tfrac1{q}(1, m_1,\dots, m_{n-1}) \times \mathbb{A}^1
\qtq{where}  m_1,\dots,m_{n-1}\in\{1, -m\}.
\end{equation}
\end{prop}

\begin{proof}

Over an open set of $\mathbb A^n_{\mathbf y}$, we have that  $(G_y)^n$ is defined as the normalisation of the variety given by the equations
$$u_i^q=y_i^{-m}+t\qtq{for }i=1,\dots,n.$$
Let $Y_0,\dots,Y_n$ be homogeneous coordinates on $\mathbb P^n$ so that $y_i=Y_i/Y_0$ for $i=1,\dots,n$.
In order to extend the equations above to $\mathbb P^n$, we write $m=aq-r$ where $0<r<q$, we multiply $u_i^q=y_i^{-m}+t$ by $Y_i^{aq}$ and set
$V_i:=u_iY_i^a$. Thus, we get the equations
\begin{equation}\label{e_0}
V_i^q=Y_0^mY_i^r+tY_i^{aq} \qtq{for} i=1,\dots,n.
\end{equation}
We can view these as defining a complete intersection subvariety
$$
\mathcal{X}\subset \mathbb P(1^{n+1}_{\mathbf y}, a^n_{\mathbf v})\times \mathbb A^1_t.
$$
Then $\mathcal{Y}$ is the normalisation of $\mathcal{X}$.

Since $\mathcal{Y}$ is smooth if $Y_0\neq 0$, we are only interested in  the singularities along $\{Y_0=0\}$.
By symmetry, it is enough to consider the affine chart $Y_n\neq 0$. Let $\tilde{y}_i = \frac{Y_i}{Y_n}$ and $v_i = u_i\tilde{y}_i$.
Then, for $i=n$, \eqref{e_0} becomes  $v_n^q=\tilde{y}_0^m+t$. Thus, we can eliminate $t$ from the other equations in \eqref{e_0} and get
$$
v_i^q=\tilde{y}_0^m\tilde{y}_i^r+v_n^q\tilde{y}_i^{aq}-  \tilde{y}_0^m \tilde{y}_i^{aq}  \qtq{for} i=1,\dots,n-1$$
in $\mathbb{A}^n_{\tilde{y}_0, \cdots \tilde{y}_{n-1}} \times \mathbb{A}^n_{v_1, \cdots, v_n}$.
Setting $\bar v_i=v_i-v_n\tilde{y}_i^a$ yields
\begin{equation}\label{e_1}
\bar v_i^q=\tilde{y}_0^m\tilde{y}_i^r(1-\tilde{y}_i^m) \qtq{for}  i=1,\dots,n-1
\end{equation}
in $\mathbb{A}^n_{\tilde{y}_0, \cdots \tilde{y}_{n-1}} \times \mathbb{A}^{n-1}_{\bar{v}_1, \cdots, \bar{v}_{n-1}} \times \mathbb{A}^1_{v_n}$.

Pick any point $p_0$.
If $1-\tilde{y}_i^m$ is nonzero at $p_0$ then $\bar y_i:=\tilde{y}_i\sqrt[r]{1-\tilde{y}_i^m}$ is an \'etale coordinate, and if $1-\tilde{y}_i^m$ vanishes at $p_0$ then  $\bar y_i:=\tilde{y}_i^r(1-\tilde{y}_i^m)$ is.
Let $\bar{y}_0=\tilde{y}_0$. Thus, at every point we can choose  \'etale coordinates
such that the equations \eqref{e_1} become
\begin{equation}\label{only.quotient.say}
\bar v_i^q=\bar{y}_0^m\bar y_i^{r_i} \qtq{for} i=1,\dots,n-1
\end{equation}
which define $\mathcal{Z} \subset \mathbb{A}^n_{\bar{y}_0, \cdots \bar{y}_{n-1}} \times \mathbb{A}^{n-1}_{\bar{v}_1, \cdots, \bar{v}_{n-1}} \times \mathbb{A}^1_{v_n}$
where $r_i\in\{r,1\} $.

Let $\mathbf z=(z_0,\dots,z_{n-1})$. The variety $\mathcal{Z} $ 
admits a purely inseparable finite cover $\mathbb A^n_{\mathbf z}\times \mathbb{A}^1_{v_n}\to \mathcal{Z}$ of degree $q$ defined by
\begin{equation}\label{quotient}\bar{y}_i=z_i^q\ \qtq{for } i=0,\dots,n-1 \qtq{and}  \bar{v}_i=z_0^mz_i^{r_i}\ \quad\text{for }i=1,\dots,n-1.
\end{equation}
Set $m_i=1$ if $r_i=r$ and  $m_i=r$ if $r_i=1$, which guarantees
\begin{equation}\label{eq:div}q| m+ r_im_i ~\  \mathrm{ for} ~i=1,\dots,n-1.\end{equation}
To prove the proposition, it is enough to show that the normalisation
of the variety defined by (\ref{only.quotient.say}) coincides with the toric quotient variety
$$\mathcal{Z}':=\mathbb A^n_{\mathbf z}\bigl\slash \tfrac1{q}(1, m_1,\dots, m_{n-1}) \times \mathbb{A}^1_{v_n}$$
where the first factor is defined by the fan $\Sigma':= \sum_{i=1}^{n} \mathbb{R}_{\geq 0}\cdot e_i\subset N'\otimes \mathbb R$ where
$$N':=\tsum_{i=1}^{n}\mathbb{Z}\cdot e_i + \mathbb{Z}\cdot\tfrac1{q}(e_1+ m_1e_2 +\dots, m_{n-1}e_n).$$
To see this, we regard $\mathbb{A}_{\mathbf z}^n$ as the toric variety defined by the fan
$$\Sigma:=\tsum_{i=1}^n\mathbb R_{\ge 0}e_i \subset N\otimes \mathbb R=N'\otimes \mathbb R~\ \mathrm{where}~N:=\tsum_{i=1}^n\mathbb Z\cdot e_i,$$
then the natural map $\eta: \mathbb{A}_{\mathbf z}^n \to \mathbb A^n_{\mathbf z}\bigl\slash \tfrac1{q}(1, m_1,\dots, m_{n-1})$ corresponds to the natural inclusion of polyhedral cones $\Sigma\to \Sigma'$. By  (\ref{eq:div}), we have that $\bar{y}_i, \bar{v}_i$ are all regular functions of $\mathbb A^n_{\mathbf z}\bigl\slash \tfrac1{q}(1, m_1,\dots, m_{n-1})$, thus the quotient morphism (\ref{quotient}) factors as
$$\mathbb{A}_{\mathbf z}^n \times \mathbb{A}^1_{v_n} \to  \mathcal{Z}' \to \mathcal{Z}$$
where $\mathcal{Z}' \to \mathcal{Z}$ is birational and finite. Since $\mathcal{Z}'$ is normal, the claim follows.
\end{proof}

\begin{prop}[Reid-Tai criterion]\label{rtc}
Consider the quotient toric singularity of the form
$$
  \mathbb A^n\bigl\slash \tfrac1{q}(m_0, m_1,\dots, m_{n-1}),
$$
where $p\nmid m_i$, for each $i=0,\dots,n-1$.

Then the singularity is canonical (resp.\ terminal) if and only if, for every $j=1,\dots, q-1$, we have
$$
\tsum_{i=0}^{n-1}\ \overline{jm_i} \geq q  \qtq{(resp.\ $>q$),}
$$
where   $\overline{c}$ denotes the residue modulo $q$.
\end{prop}
Though this is usually stated in characteristic 0 (e.g. see \cite[p. 376]{Reid85}), these singularities have toric resolutions and the
discrepancies depend only on the toric fan, hence independent of the characteristic.

\medskip

 \begin{cor}\label{only.quotient.say.cor}
$Y(n,m,q)$ has canonical singularities for $n\geq q$ and terminal singularities for $n>q$.
\end{cor}
\begin{proof}
Since $p\nmid m_i$ we have that  $\overline{jm_i}\geq 1$, for all $i=1,\dots,n-1$ and $j=1,\dots,q-1$. In particular, by Proposition \ref{rtc}, we obtain that
the singularity \eqref{e_quotient} is  canonical for  $n\geq q$ and terminal for $n>q$.
The same claims hold true for the generic fibre $Y(q,m,n)$ of $\mathcal{Y} \to \mathbb{A}^1_t$.
\end{proof}
Note that these bounds are sharp if all $m_i=1$, but otherwise   smaller values of $n$ would work.

\medskip
\begin{prop}\label{p_fano}
If $n \geq m(q-1)$ then $Y(n,m,q)$ is a Fano variety with canonical singularities.
\end{prop}
\begin{proof}
As in the proof of Proposition \ref{p_only.quotient.say}, we have that the non-normal locus of $Y_{\bar{K}}$ maps over $H = \{Y_0=0\} \subset \mathbb{P}_{\bar{K}}^n$. Let $\psi\colon Z_{\bar{K}} \to \mathbb{P}_{\bar{K}}^n$ be the induced morphism and let  $\bar{H}= \psi^{-1}H$. Since $\bar{H}$ is irreducible, if $\nu\colon  Z_{\bar{K}} \to Y_{\bar{K}}$ is the normalisation map, we have
$$\nu^*K_{Y_{\bar{K}}} \sim K_{Z_{\bar{K}}}+c\bar{H}$$
where $c$ is an integer depending on  the contribution of the conductor.
We claim that $c= m(q-1)$, so that, since $Z_{\bar{K}} \cong \mathbb{P}_{\bar{K}}^n$, we have
$$
\nu^*K_{Y_{\bar{K}}} \sim K_{Z_{\bar{K}}}+(mq-m)\bar{H} \sim (mq-m-n-1)\bar{H},
$$
which, in turn,  implies that if $m(q-1)\leq n$ then $X$ is Fano and, by Corollary \ref{only.quotient.say.cor}, it has canonical singularities.

To compute $c$ we may restrict to an \'{e}tale open neighbourhood of the generic point of $\bar{\phi}^{-1}(H)$.
We work on the normalisation $\mathcal{U}$ of the chart (\ref{only.quotient.say}), which is an \'{e}tale open subset of $\mathcal{Y}$. By shrinking $\mathcal{U}$, we may assume that $\bar{y}_i \neq 0$ and $r_i=r$ for all $i=1,\dots, n-1$. Thus, $m_i =1$ for any $i$ and $\mathcal{U}$ is regular, but not smooth, over $\mathbb{A}^1_t$. Note that $\mathcal{U}$ is an open subset of the affine chart
$$
{\rm Spec}~ k\bigl[w_0=\frac{z_0}{z_1}, w_1=z_1^q, w_2=\frac{z_2}{z_1},\dots,w_{n-1}=\frac{z_{n-1}}{z_1}, v_n\bigr].
$$
Recall the equation  $v_n^q -t =\bar{y}_0^m = w_0^{mq}w_1^m$.
A generator of $\omega_Y$ on $U = \mathcal{U}_K$ is
$$
{w_0^{-mq}w_1^{1-m}}\cdot {d v_n\wedge d w_0 \wedge dw_2 \wedge \cdots\wedge dw_{n-1}}.
$$
Geometrically, set  $\bar v_n:=\bigl(v_n-t^{1/q}\bigr)  w_0^{-m}$ to get
$$
\bar v_n^q=w_1^m,  \qtq{which is smooth if} w_1\neq 0.
$$
A generator of $\omega_{Z_{\bar{K}}}$ over $\mathcal{U}_{\bar{K}}$ is
$$
{w_1^{1-m}}\cdot {d \bar v_n\wedge d w_0\wedge dw_2 \wedge \cdots dw_{n-1}} =  w_0^{-m}{w_1^{1-m}}\cdot {d v_n\wedge d w_0\wedge dw_2 \wedge \dots \wedge dw_{n-1}}.
$$
Since the divisor $\psi^{-1}(H)$ is defined by $\{w_0 =0\}$, we get $c=m(q-1)$ as  claimed.
\end{proof}

\subsection{Counterexamples to subadditivity of Kodaira dimensions}\label{sec:construction2}
Let $G$ be the curve over $K=k(t)$ constructed in the previous subsection. This yields a surface $S_1$ over $k$ equipped with a fibration $S_1 \to \mathbb{P}^1$. Consider a separable base change $C \to \mathbb{P}^1$ from some smooth projective curve $C$ with $g(C) \geq 2$. Let $S$ be a smooth resolution of $S_1 \times_{\mathbb{P}^1} C$. Then the generic fibre $F$ of the fibration $f\colon S \to C$ coincides with $G\otimes_{K}K(C)$ and $\kappa(S) \geq \kappa(C) + \kappa(F) \geq 1$.
Let $X^{(n)}:=S \times_C \cdots \times_C S$ be the $n$-th fibre product of $S$ over $C$. If $n \geq m(q-1)$ then there exists a birational model $Y^{(n)} \to C$ such that the generic fibre is isomorphic to $Y(n,m,q)$, and such that,  by Proposition \ref{p_fano}, we have that  $\kappa(Y^{(n)}, K_{Y^{(n)}}) = -\infty$. Similarly to Section \ref{sec:construction1}, this yields examples of dimension $n+1$.

More specifically, by setting $m=2, q=p$ we get examples of dimension $2p-1$, and by setting $m=3, q=p=2$ we get examples of dimension 4.

\section{Counterexamples to a log version of $C_{2,1}$ over imperfect fields}\label{sec:non-alg-closed-field}
We conclude this paper by introducing an example of a fibration that violates a log version of Question~\ref{iitaka-char-p} over an imperfect field.
The authors learned about this example from H. Tanaka.
\begin{eg} \label{eg:tanaka}
Let $k$ be an imperfect field of characteristic $p\in\{2,3\}$.
Tanaka~\cite[Theorem~1.4]{Ta16} constructed a $k$-morphism $\rho\colon S \to C$ with the following properties:
\begin{itemize}
\item $S$ is a projective regular surface over $k$;
\item there exists a prime divisor $C_2$ on $S$ such that if $\Delta_S:=\left( \frac{2}{p} -\varepsilon \right)C_2$, where $\varepsilon \in \mathbb Q$ is such that $0 < \varepsilon \ll \frac{2}{p}$, then  $(S, \Delta_S)$ is Kawamata log terminal and $-(K_S +\Delta_S)$ is ample;
\item $\rho$ is a $\mathbb P^1$-bundle;
\item $C$ is a projective regular curve over $k$ with $K_C \sim 0$.
\end{itemize}
These properties can be viewed as pathologies in birational geometry in positive characteristic. Indeed, it is known that the image of a variety of log Fano type in characteristic zero is again of log Fano type (\cite[Lemma~2.8]{PS09}).

Set $\Gamma:=\frac{2}{p}C_2$.
Then the pair $(S,\Gamma)$ is log canonical and \cite[Lemma~3.4]{Ta16} implies that $-(K_S +\Gamma)$ is $\mathbb Q$-linearly equivalent to the pullback of an ample $\mathbb Q$-divisor on $C$.
Hence, if $S_{\eta}$ is the generic fibre of $\rho$, then
$$
-\infty
= \kappa(S, K_S +\Gamma)
< \kappa(S_{\eta}, K_{S_{\eta}} +\Gamma|_{S_{\eta}})
+ \kappa(C, K_C) =0.
$$
\end{eg}
\begin{rem} \label{rem:tanaka}
We use the same notation as in Example~\ref{eg:tanaka}.

\noindent (1)
The pair $\left( S_{\overline\eta}, \Gamma|_{S_{\overline\eta}} \right)$ is not $F$-pure, where $S_{\overline\eta}$ is the geometric generic fibre of $\rho$. Indeed, by the construction in \cite[\S 3]{Ta16},  there exists  a closed point $Q \in S_{\overline\eta}$ such that $C_2|_{S_{\overline\eta}} = pQ$ as divisors.

\noindent (2)
Roughly speaking, the curve $C$ is unirational.
More precisely,  \cite[Proposition~3.6]{Ta16} implies that
there is an extension $k'/k$ of degree $p$ such that the normalisation of $C\times_k k'$ is $k'$-isomorphic to the projective line $\mathbb P^1_{k'}$.
\end{rem}
\bibliographystyle{plain}

\end{document}